\documentclass[12pt,reqno]{amsart}

\usepackage{amsaddr}
\usepackage{url}
\usepackage{bm}
\usepackage{amssymb}

\newcommand\bi{\bm i}
\newcommand\bj{\bm j}
\newcommand\bk{\bm k}

\newtheorem{theorem}{Theorem}
\newtheorem{lemma}[theorem]{Lemma}
\newtheorem{cor}[theorem]{Corollary}
\newtheorem{prop}[theorem]{Proposition}

\begin{document}

\title{Functional calculus for dual quaternions}
\author{Stephen Montgomery-Smith}
\address{Department of Mathematics, University of Missouri, Columbia, MO 65211, USA\\
\rm\url{stephen@missouri.edu}\\
\rm\url{https://stephenmontgomerysmith.github.io}\\
}

\begin{abstract}  We give a formula for $f(\eta)$, where $f :\mathbb C \to \mathbb C$ is a continuously differentiable function satisfying
$f(\bar z) = \overline{f(z)}$,
and $\eta$ is a dual quaternion.  Note this formula is straightforward or well known if $\eta$ is merely a dual number or a quaternion.  If one is willing to prove the result only when $f$ is a polynomial, then the methods of this paper are elementary.

This preprint has not undergone peer review (when
applicable) or any post-submission improvements or corrections. The Version of Record of this
article is published in Adv. Appl. Clifford Algebras, and is available online at \url{https://doi.org/10.1007/s00006-023-01282-y}.

Added August 10, 2026.  The author recently came across a new approach by `MathFish.'  It is described at the end of this document, and is not available in the Version of Record.
\end{abstract}



\maketitle

A \emph{quaternion} is a quadruple of real numbers, written as $A = w + x \bi + y \bj + z \bk$, with the algebraic operations $\bi^2 = \bj^2 = \bk^2 = \bi \bj \bk = -1$.  Its \emph{conjugate} is $\bar A = A = w - x \bi - y \bj - z \bk$.  Its \emph{norm} is $|A| = (w^2+x^2+y^2+z^2)^{1/2}$.  It is called a \emph{pure} quaternion if $w = 0$.  We identify 3-vectors with pure quaternions, by identifying $\bi$, $\bj$, and $\bk$ with the three standard unit vectors.  A \emph{dual number} is a pair of real numbers, written as $\alpha = a + \epsilon b$, with the algebraic operation $\epsilon^2 = 0$.  A \emph{dual quaternion} is a pair of quaternions, written as $\eta = A + \epsilon B$, again with $\epsilon^2 = 0$.  The \emph{conjugate} of this dual quaternion is $\bar \eta = \bar A + \epsilon \bar B$.  The notion of dual quaternion goes back to Clifford \cite{clifford}.  For more information, including how they are used by the graphics card industry and in robotics, we refer the reader to \cite{adorno,agrawal,han-et-al,kavan-et-al,kavan-et-al-2,kenwright,kussaba-et-al,schilling1,schilling2,wiki-2,yang-et-al}.

The symbol $i$ denotes one of the square roots of $-1$, and in this paper won't be identified with $\bi$.

We say that a function $f:\Omega \to \mathbb C$ is \emph{valid} if $\Omega$ is a subset of $\mathbb C$ such that $z\in\Omega \Leftrightarrow \bar z \in \Omega$, and
\begin{equation}
f(\bar z) = \overline{f(z)} \quad \text{for $z \in \Omega$.}
\end{equation}
Note that if $f(x+iy)$ is a polynomial in $x$ and $y$:
\begin{equation}
f(x+iy) = \sum_{m,n} r_{m,n} x^m (iy)^n.
\end{equation}
then $f$ is valid if and only if the $r_{m,n}$ are real.

We show how to extend the definition of $f$, at least when $f$ is appropriately smooth, to dual quaternions in such a way that it is correct for polynomials.  In all of the statements of our results, when we say that a class of functions extends to a subset of dual quaternions, we mean the following standard definition of a functional calculus.
\begin{enumerate}
\item The class of functions contains $f(z) = 1$, $f(z) = z$, and if $f$ and $g$ are in this class, and $a$ is a real number, then $f+g$, $fg$, $\bar f$, and $a f$ are also in this class.
\item Given any dual quaternion $\eta$ from the prescribed subset, there is a map from the class to dual quaternions, $f \mapsto f(\eta)$, such that $f(z) = 1$ maps to $1$, $f(z) = z$ maps to $\eta$, $f+g$ maps to $f(\eta) + g(\eta)$, $fg$ maps to $f(\eta) g(\eta)$, $\bar f$ maps to $\overline{f(\eta)}$, and $a f$ maps to $a f(\eta)$.
\item If any topology is specified on the class of functions, then the map $\eta \mapsto f(\eta)$ is continuous.
\end{enumerate}
All of the results of this paper could be verified for polynomials simply by looking at the individual monomials, and indeed this is how Lemmas~\ref{f(d)} and~\ref{f(dq) commute} are best proved.  Another method is to verify the results if $f(z) = 1$ and $f(z) = z$, and show that the set of functions for which it is true is closed under addition, multiplication, conjugation, and multiplication by real numbers.  However, for Lemma~\ref{f(dq) anti commute}, these straightforward approaches are nevertheless quite mysterious, and we believe that our approach is more intuitive.

Once the results are verified for polynomials, we can extend to the wider class of functions using the Stone-Weierstrass Theorem \cite{rudin}.

We should mention that the formula for dual numbers is well known \cite{wiki-1}, and indeed we cite it as Lemma~\ref{f(d)}.

We also want to mention the remarkable, and rather different, approach taken by Selig  to this problem \cite{selig}.

If $f$ is continuously differentiable, denote
\begin{equation}
f_x(x+iy) = \dfrac{\partial}{\partial x} f(x+iy) , \quad
f_{iy}(x+iy) = -i\dfrac{\partial}{\partial y} f(x+iy) .
\end{equation}
(Thus the Cauchy-Riemann conditions can be stated as $f$ is analytic if and only if $f_x = f_{iy}$.)  Define the following real valued, continuous, functions, which are even in $y$:
\begin{align}
\label{g}
g(x+iy) &= \dfrac{f(x+iy) + f(x-iy)} 2, \\
\label{h}
h(x+iy) &= \begin{cases}
\dfrac{f(x+iy)-f(x-iy)}{2iy} &\text{if $y \ne 0$} \\
f_{iy} (x) &\text{if $y = 0$} ,
\end{cases}
\end{align}
so that
\begin{equation}
\label{f=g+iyh}
f(x+iy) = g(x+iy) + i y h(x+iy) .
\end{equation}

First we state how to extend $f$ to the quaternions.  The proof is straightforward, because any quaternion that is a unit 3-vector behaves formally exactly like $i$ in $\mathbb C$.

\begin{theorem}
\label{f(q)}
Suppose that $f:\Omega \to \mathbb C$ is a valid function.  Then $f$ extends to quaternions
\begin{equation}
f(a_0 + \bm a_1) = \begin{cases}
g(a_0+i|\bm a_1|) +h(a_0+i|\bm a_1|) \bm a_1 & \text{if $\bm a_1 \ne 0$} \\
f(a_0) & \text{if $\bm a_1 = 0$}, \end{cases}
\end{equation}
where $a_0$ is real and $\bm a_1$ is a 3-vector.  Furthermore the norm is preserved in that
\begin{equation}
| f(a_0 + \bm a_1) | = | f(a + i |\bm a_1|) | .
\end{equation}
\end{theorem}

If $f$ is a valid polynomial, then a formula that extends $f$ to all dual quaternions is essentially stated in \cite{wiki-1}.  If $A$ and $B$ are quaternions, then
\begin{equation}
\label{proto f(dq)}
f(A + \epsilon B) = f(A) + \epsilon \left.\frac d{dr} f(A + rB) \right|_{r=0}.
\end{equation}
If $A$ and $B$ commute, this immediately implies Lemma~\ref{f(dq) commute} below.  But if $A$ and $B$ do not commute, it is not immediately apparent how to use this formula.  Thus we now state the main result of this paper.

\begin{theorem}
\label{f(dq)}
Let $f:\Omega \to \mathbb C$ be a valid continuously differentiable function, where $\Omega$ is open in $\mathbb C$.  Define $h$ by equation~\eqref{h}.  Then $f$ can be extended to a continuous function on all dual quaternions as follows.  Given quaternions $A$ and $B$, decompose $A = a_0 + \bm a_1$ and $B = b_0 + \bm b_1 + \bm b_2$, where $a_0$ and $b_0$ are real, $\bm a_1$, $\bm b_1$ and $\bm b_2$ are 3-vectors, $\bm b_1$ is parallel to $\bm a_1$, and $\bm b_2$ is perpendicular to $\bm a_1$.  Write $B_1 = b_0 + \bm b_1$.  If $a + i |\bm a_1| \in \Omega$, then
\begin{align}
f(A + \epsilon B)
&= g(A) 
+ \epsilon f_x(A) b_0
+ h(A) (1 + \epsilon \bm b_2)
+ \epsilon f_{iy}(A) \bm b_1 \\
&= f(A) 
+ \epsilon f_x(A) b_0
+ \epsilon f_{iy}(A) \bm b_1
+ \epsilon h(A) \bm b_2 .
\end{align}
\end{theorem}

Before proving this result, let us provide some examples.  First something simple.

\begin{cor}
With the hypotheses of Theorem~\ref{f(dq)} we have
\begin{align}
\overline{A + \epsilon B} &= \overline{A} + \epsilon \overline{B} \\
|A + \epsilon B| &= |A| + \epsilon \frac{\text{\rm Re}(A \overline B)}{|A|} \quad (A \ne 0) .
\end{align}
\end{cor}

Next, we compute the exponential function.  Formulas for the exponential and logarithm are given in \cite{wang-et-al}, but we believe ours are more explicit.  Another formula for the exponential and the logarithm is given in \cite{selig}, with a correction for the logarithm in \cite{wu-et-al}.  While their formula gives the same result for the exponential, this is not immediately obvious, and we haven't checked for the logarithm.  A formula for the exponential and logarithm is also given in \cite{han-et-al}, but we believe that their formula only works if $A$ and $B$ commute.

\begin{cor}
\label{exp dq}
With the hypotheses of Theorem~\ref{f(dq)} we have
\begin{multline}
\label{exp theta}
\exp(A + \epsilon B ) \\
= e^{a_0} \left(\left(\cos(|\bm a_1|) + \dfrac{\sin(|\bm a_1|)}{|\bm a_1|} \bm a_1 \right) \left(1 + \epsilon B_1\right) + \epsilon\frac{\sin(|\bm a_1|)}{|\bm a_1|} \bm b_2\right),
\end{multline}
where if $\bm a_1 = 0$, we set $\sin(|\bm a_1|)/|\bm a_1| = 1$.
\end{cor}

We can compute the logarithm in the same way, but then we can only get the principal value.  Instead we define a logarithm as a right inverse to the exponential function: $\exp(\log(\eta)) = \eta$.

\begin{cor}
\label{log dq}
Assume the hypotheses of Theorem~\ref{f(dq)}, with $A \ne 0$.  Let $t$ be the angle for a choice of polar coordinates for $(a_0, |\bm a_1|)$.  If $\bm a_1 \ne 0$, then a choice of $\log(A+\epsilon B)$ is
\begin{equation}
\log(|A|) + \dfrac t{|\bm a_1|} (\bm a_1 + \epsilon \bm b_2) + \epsilon \dfrac1{|A|^2} \bar A_1 B_1 .
\end{equation}
In the case $\bm a_1 = 0$, and $a_0 \ne 0$, we can assume $t = n \pi$ for some integer $n$, and $\bm b_2 = 0$.  Let $\bm p$ be any 3-vector perpendicular to $\bm b_1$ if $n \ne 0$, and $\bm p = 0$ if $n = 0$.  Then a choice of $\log(A + \epsilon B)$ is
\begin{equation}
\label{t=n pi}
\log(|a_0|) + n \pi \dfrac{\bm b_1}{|\bm b_1|} + \epsilon \dfrac 1{a_0} B_1 + \epsilon \bm p,
\end{equation}
where if $\bm b_1 = 0$, we interpret $\bm b_1 / |\bm b_1|$ as any unit 3-vector.
\end{cor}

Next we give a formula for powers.  For $\alpha \in \mathbb R$ and $z \in \mathbb C \setminus (-\infty,0]$, let $z^\alpha$ denote the principle value.

\begin{cor}
With the hypotheses of Theorem~\ref{f(dq)}, if $A \notin (-\infty,0]$, then we have
\begin{equation}
(A + \epsilon B )^\alpha =
A^\alpha + \epsilon \alpha A^{\alpha - 1} B_1 + \epsilon \frac{\text{\rm Im}(A^\alpha)}{\text{\rm Im}(A)} \bm b_2 ,
\end{equation}
where if $\text{\rm Im}(A) = 0$, we set ${\text{\rm Im}(A^\alpha)}/{\text{\rm Im}(A)}$ to $\alpha A^{\alpha - 1} $.
\end{cor}

Note that for $\alpha = -1$, it can be shown that this is equivalent to the formula
\begin{equation}
(A + \epsilon B)^{-1} = A^{-1} - \epsilon A^{-1} B A^{-1} \quad (A \ne 0).
\end{equation}

Finally, we reproduce Selig's formula for the Cayley Transform \cite{selig}.

\begin{cor}
With the hypotheses of Theorem~\ref{f(dq)}, if $A \ne 1$, then we have
\begin{equation}
\begin{aligned}
\frac{1 + A + \epsilon B}{1 - (A + \epsilon B)}
&=
\frac{1+A}{1-A} + \epsilon \frac{2}{(1-A)^2} \bm b_1 + \epsilon \frac{2}{|1-A|^2} \bm b_2 \\
&=
\frac{1+A}{1-A} + 2 \epsilon \frac{1}{(1-A)} B \frac{1}{(1-A)} .
\end{aligned}
\end{equation}
\end{cor}

Now we start the proof of Theorem~\ref{f(dq)}.  As stated above, these first two results come straight from equation~\eqref{proto f(dq)}.  Lemma~\ref{f(d)} is found in \cite{wiki-1}.

\begin{lemma}
\label{f(d)}
Let $f:\Omega \to \mathbb R$ be a continuously differentiable function, where $\Omega$ is open in $\mathbb R$.  Then $f$ extends to dual numbers $a + \epsilon b$ with $a \in \Omega$ by
\begin{equation}
\label{dual number case}
f(a + \epsilon b ) = f(a) + \epsilon f'(a) b_0 .
\end{equation}
\end{lemma}

\begin{lemma}
\label{f(dq) commute}
Let $f:\Omega \to \mathbb C$ be a valid continuously differentiable function, where $\Omega$ is open in $\mathbb C$.  Then $f$ extends to dual quaternions $A + \epsilon B$ where $A$ and $B$ commute with
\begin{equation}
\label{commuting case}
f(A + \epsilon B ) = f(A) + \epsilon f_x(A) b_0
+ \epsilon f_{iy}(A) \bm b_1 ,
\end{equation}
where $B = b_0 + \bm b_1$, $b_0$ is real, and $\bm b_1$ is a 3-vector.
\end{lemma}

\begin{lemma}
\label{f(dq) anti commute}
Let $f:\Omega \to \mathbb C$ be a valid continuously differentiable function, where $\Omega$ is an open subset of $i\mathbb R$.  Define the continuous real valued function
\begin{equation}
\label{h ident 2}
h(iy) = \begin{cases}
\dfrac{f(iy) - f(-iy)}{2iy} &\text{if $y \ne 0$} \\
f_{iy}(0) &\text{if $y = 0$} .
\end{cases}
\end{equation}
Then $f$ extends to dual quaternions of the form $\eta = \bm a + \epsilon \bm b_1 + \epsilon \bm b_2$, where $\bm a$, $\bm b_1$, and $\bm b_2$ are pure, $\bm b_1$ and $\bm a$ are parallel, and $\bm b_2$ and $\bm a$ are perpendicular, by the formula
\begin{equation}
\label{f(dq) anti commute equ}
f(\eta) = f(\bm a + \epsilon \bm b_1) + \epsilon h(\bm a) \bm b_2 ,
\end{equation}
\end{lemma}

\begin{proof}  Without loss of generality, we may assume that $f$ is a polynomial.  Define the function
\begin{equation}
g(iy) = f(iy) + f(-iy) .
\end{equation}
Since $g$ and $h$ are even polynomials, we can create polynomials $k$ and $l$ on an open subset of $[-\infty,0]$ by
\begin{equation}
k(z) = g(\sqrt z), \quad l(z) = h(\sqrt z) .
\end{equation}
Then we have the identities
\begin{equation}
\label{f g h k l ident}
f(z) = k(z^2) + l(z^2) z , \quad k(z^2) = g(z), \quad l(z^2) = h(z) .
\end{equation}
Let $\alpha = \bm a + \epsilon \bm b_1$, and $\beta = \epsilon\bm b_2$.  Use equation~\eqref{proto f(dq)} to replace $z$ with $\eta$ in equations~\eqref{f g h k l ident}.  Note that $\eta^2 = \alpha^2$.  Hence
\begin{equation}
\begin{aligned}
f(\eta)
&= k(\eta^2) + l(\eta^2) \eta = k(\alpha^2) + l(\alpha^2) \eta \\
&= g(\alpha) + h(\alpha) (\alpha + \beta) = f(\alpha) + h(\alpha) \beta.
\end{aligned}
\end{equation}
\end{proof}

\begin{proof}[Proof of Theorem~\ref{f(dq)}]
By Lemma~\ref{f(dq) commute}, we only need to show
\begin{equation}
f(A+\epsilon B) = f(A + \epsilon B_1) + \epsilon h(A) \bm b_2 .
\end{equation}
Without loss of generality, assume $f$ is a polynomial.  Define the functions $u,v:(\Omega-a_0) \cap i\mathbb R \to \mathbb C$ by
\begin{align}
u(iy) &= f(a_0 + iy) , \\
v(iy) &= f_{x}(a_0 + iy) b_0,
\end{align}
so that by Lemma~\ref{f(d)}, applied to the real and imaginary parts of $x \mapsto f(x + iy)$, and replacing $x$ by $a_0 + \epsilon b_0$, we have
\begin{equation}
f(a_0 + \epsilon b_0 + i y) = u(iy) + \epsilon v(iy) .
\end{equation}
Now apply Lemma~\ref{f(dq) anti commute} to $u$ and $v$.  (A second derivative of $f$ appears when applying the lemma to $v$, but this disappears in the final result because $\epsilon^2 = 0$.)
\end{proof}

During the writing of this paper, we came across a different approach to proving Theorem~\ref{f(dq)} motivated by Chasles' Theorem and the approach taken in \cite{wang-et-al}.  It can be shown that any dual quaternion may be factored as
\begin{equation}
a + \bm a_1 + \epsilon (b + \bm b_1) = (1 + \epsilon \bm r) (a + \bm a_1 + \epsilon \tilde B) (1 - \epsilon \bm r),
\end{equation}
where $\tilde B$ commutes with $\bm a_1$, and
\begin{equation}
\bm r = \begin{cases} \frac{\bm b_1 \bm a_1 - \bm a_1 \bm b_1 }{4|\bm a_1|^2} & \text{if $\bm a_1 \ne 0$} \\ 0 & \text{if $\bm a_1 = 0$} \end{cases}
\end{equation}
is a 3-vector.  (A similar formula is also in \cite{selig}).  Also, if $f:\Omega \to \mathbb C$ is a valid polynomial, then
\begin{equation}
f\bigl((1 + \epsilon \bm r) (A + \epsilon B) (1 - \epsilon \bm r)\bigr) = (1 + \epsilon \bm r) f(A + \epsilon B) (1 - \epsilon \bm r) .
\end{equation}
In this way, the computation is reduced to Lemma~\ref{f(dq) commute}.

\bigskip

Finally, we note that the methods of the proof of Lemma~\ref{f(dq) anti commute} give a cute, albeit simple, result for the Dunford-Riesz functional calculus of analytic functions on complex algebras \cite{dunford-et-al}.

\begin{prop}
\label{analytic}
Suppose $f:\Omega \to \mathbb C$ is an analytic function on an open subset $\Omega$ of $\mathbb C$ satisfying $z \in \Omega \Leftrightarrow -z \in \Omega$.  Define the analytic functions on $\{z^2 : z \in \Omega\}$ by
\begin{align}
\label{g a}
k(z) &= \dfrac{f(\sqrt z) + f(-\sqrt z)}{2} ,\\
\label{h a}
l(z) &= \begin{cases}
\dfrac{f(\sqrt z) - f(-\sqrt z)}{2\sqrt z} &\text{if $z \ne 0$} \\
f'(0) &\text{if $z = 0$} ,
\end{cases} \\
\end{align}
If $A$ and $B$ are anti-commuting elements of a complex algebra such that their spectra $\sigma(A) \cup \sigma(B) \cup \sigma(A+B) \subset \Omega$, then
\begin{equation}
f(A+B) = k(A^2+B^2) + (A+B) l(A^2+B^2) .
\end{equation}
\end{prop}

Applying this to $(A+B)^n$, where $n$ is a non-negative integer, gives the coefficients of the so called Pauli-Pascal triangle \cite{horn,sloane}.

\section*{Another proof of Theorem~\ref{f(dq)}}

This new proof is built off the ideas of `MathFish' \cite{math-fish}, who gave this proof in the special case of the exponential and logarithm functions.  First, we have the following easy result, which is a very mild extension of Lemma~\ref{f(d)}.

\begin{lemma}
\label{f(dc)}
Let $f:\Omega \to \mathbb C$ be a continuously differentiable function, where $\Omega$ is open in $\mathbb C$.  Then $f$ extends to dual complex numbers as
\begin{equation}
f(z_0 + \epsilon z_1 ) = f(z_0) + \epsilon f_x(z_0) \text{\rm Re}(z_1)
+ i \epsilon f_{iy}(A) \text{\rm Im}(z_1) .
\end{equation}
\end{lemma}

Some points to make.  If $u:\Omega\to\mathbb R$ is continuously differentiable, valid, and real valued and $A = a_0 + \bm a_1$ is a quaternion, since the function $y \mapsto u(x + iy)$ is even, from Definition~\eqref{g},
\begin{equation}
u(A) = u(a_0 + i |\bm a_1|).
\end{equation}

If $v:\Omega\to i\mathbb R$ is continuously differentiable, valid and imaginary valued, then since the function $y \mapsto v(x + iy)$ is odd, from Definition~\eqref{h}, if $\bm a_1 \ne 0$,
\begin{equation}
v(A) = -i v(a_0 + i |\bm a_1|) \frac{\bm a_1}{|\bm a_1|}.
\end{equation}
Some care needs to be taken with this expression, since we are combining complex numbers and quaternions in the same equation.  It should be understood that the multiplication $iv(z)$ takes place first, giving a real number.  Hence it is legal to multiply it against a quaternion.

Also, from Equation~\eqref{f=g+iyh}, we have
\begin{gather}
f_x(x+iy) = g_x(x+iy) + i y h_{x}(x+iy), \\
f_{iy}(x+iy) = g_{iy}(x+iy) + i y h_{iy}(x+iy) + h(x+iy),
\end{gather}
and hence
\begin{gather}
g(A) + h(A) \bm a_1 = f(A), \\
g_x(A) + h_x(A) \bm a_1 = f_x(A), \\
g_{iy}(A) + h_{iy}(A) \bm a_1 = f_{iy}(A) - h(A).
\end{gather}

\begin{proof}[New Proof of Theorem~\ref{f(dq)}] This proof is for the case $\bm a_1 \ne 0$.  The case $\bm a_1 = 0$ follows, either using a similar and easier argument, or by taking limits.

Write the dual quaternion as $X = x_0 + \bm x_1$, where $x_0 = a_0 + \epsilon b_0$, and $\bm x_1 = \bm a_1 + \epsilon (\bm b_1 + \bm b_2)$, with $\bm b_1 = r \bm a_1$.  Then $\bm x_1^2 = -|\bm x_1|^2$.  Hence by simply multiplying out polynomials and taking limits, we see that
\begin{equation}
f(x_0 + \bm x_1) = g(x_0 + i |\bm x_1|) + h(x_0 + i |\bm x_1|) \bm x_1 .
\end{equation}
Also, we have
\begin{equation}
|\bm x_1| = |\bm a_1| + \epsilon r ,
\end{equation}
where, since $\bm a_1$ and $\bm b_1$ are parallel
\begin{equation}
r = \frac{\bm a_1 \cdot \bm b_1}{|\bm a_1|} = - \frac{\bm a_1 \bm b_1}{|\bm a_1|} .
\end{equation}
Now apply Lemma~\ref{f(dc)}.
\begin{equation}
\begin{aligned}
g(x_0 + i |\bm x_1|) &= g(a_0 + i|\bm a_1| + \epsilon (b_0) + i r) )
\\&=
g(a_0 + i |\bm a_1|)
+ \epsilon g_x(a_0 + i |\bm a_1|) b_0
\\&\phantom{{}={}}
- \epsilon i g_{iy}(a_0 + i |\bm a_1|) \frac{\bm a_1 \bm b_1}{|\bm a_1|}
\\&= g(A) + \epsilon g_x(A) b_0 + g_{iy}(A) \bm b_1 .
\end{aligned}
\end{equation}
Similarly,
\begin{equation}
\begin{aligned}
h(x_0 + i |\bm x_1|) \bm x_1 &= 
h(A) \bm a_1
+ \epsilon h(A) (\bm b_1 + \bm b_2)
\\&\phantom{{}={}}
+ \epsilon h_x(A) b_0 \bm a_1
+ \epsilon h_{iy}(A) \bm b_1 \bm a_1
\end{aligned}
\end{equation}
Hence
\begin{equation}
\begin{aligned}
f(x_0 + \bm x_1) &=
g(A) + h(A) \bm a_1
\\&\phantom{{}={}}
+ \epsilon h(A) (\bm b_1 + \bm b_2)
\\&\phantom{{}={}}
+ \epsilon g_x(A) b_0
+ \epsilon h_x(A) b_0 \bm a_1
\\&\phantom{{}={}}
+ \epsilon g_{iy}(A) \bm b_1
+ \epsilon h_{iy}(A) \bm b_1 \bm a_1
\\&=
f(A) + \epsilon f_x(A) b_0 + \epsilon f_{iy}(A) \bm b_1 + \epsilon h(A) (\bm b_1 + \bm b_2 -r \bm a_1)
\\&=
f(A) + \epsilon f_x(A) b_0 + \epsilon f_{iy}(A) \bm b_1 + \epsilon h(A) \bm b_2 .
\end{aligned}
\end{equation}
\end{proof}




\end{document}